\documentclass[11pt]{amsart}

\usepackage{amssymb}

\newtheorem{thm}{Theorem}[section]
\newtheorem{prop}[thm]{Proposition}

\newtheorem{cor}[thm]{Corollary}

\theoremstyle{definition}

\theoremstyle{remark}

\theoremstyle{example}

\numberwithin{equation}{section}

\begin{document}

\title{A curvature flow in the plane with a nonlocal term }
\author{Luis A. Caffarelli}\address{Department of Mathematics, the University of Texas at Austin}\email{caffarel@math.utexas.edu}
\author{Hui Yu}
\address{Department of Mathematics, Columbia University in the City of New York}
\email{huiyu@math.columbia.edu}

\begin{abstract}
We study the geometric flow of a planar curve driven by its curvature and the normal derivative of its capacity potential. Under a convexity condition that is natural to our problem, we establish long term existence and large time asymptotics of this flow. 
\end{abstract}

 \maketitle

\tableofcontents

\section{Introduction}
In this paper we study the motion of domains in the plane, where the inward normal velocity of the boundary is given by $$V=H-u_\nu^2 \text{ on $\partial\Omega_t$}.$$ Here $\Omega_t$ is the domain at time $t$, which we assume to contain $B_1$ as a compact subset. $H$ is the curvature of its boundary, and $u_\nu$ is the normal derivative of the solution to $$\begin{cases}\Delta u=0 &\text{ in $\Omega_t\backslash\overline{B_1}$,}\\u=1 &\text{ in $B_1$,}\\ u=0 &\text{ outside $\Omega_t$.}\end{cases}$$ We call this function $u$ the capacity potential of $\Omega_t$, and often denote it by $u_{\Omega_t}$.

This dependence on the bulk introduces nonlocality to the problem as two points on $\partial\Omega_t$, far way from each other, interact through the potential. This leads to the main mathematical challenge in the study of this flow as compared to more classical geometric flows. For the motivation behind this problem and its relation to more classical geometric flows,  see \cite{Yu}. There one of the authors proved short time well-posedness under more general assumptions and in arbitrary spatial dimensions, and gave examples to show the inevitability of singularities in general, after which two possible weak formulations were also proposed. 

One of the problems left open in \cite{Yu} is to find criteria on the initial configuration $\Omega_0$ under which singularities can be avoided, and the flow remains smooth globally in time. In this paper, we give one such condition for curves in the plane, namely, \begin{equation}H> u_\nu^2 \text{ on $\partial \Omega_0$}.\end{equation}That is, on the initial domain, the curvature of the boundary strictly dominates the normal derivative of the capacity potential. 

To see why this is a natural condition, it is illustrative to look at a more classical geometric flow, the mean curvature flow, where the equation for motion is $$V=H.$$ The condition here for global existence of smooth flow is $$H\ge 0 \text{ on $\partial\Omega_0$},$$namely, the initial curve is convex. Just like our condition (1.1), this condition guarantees the curve is moving inward at least initially. Inspired by a similar result by Huisken \cite{H}, Gage-Hamilton \cite{GH} showed that this convexity is preserved by the mean curvature flow, and with this convexity they were able to give uniform curvature bound and show the existence of a global smooth flow. Both Huisken \cite{H} and Gage-Hamilton \cite{GH} rely crucially on the fact that the mean curvature flow gives rise to explicit evolution equations for geometric quantities like the length of the curve, the area and  inscribe radius of the enclosed region, curvature of the curve and even ratio between curvature at different points on the curve.   

Due to the nonlocality introduced by the bulk term $u_\nu$, except for very special configurations, deriving explicit evolution for these quantities is very difficult if not impossible.  Nevertheless, ideas from geometric flows and free boundary problems, and the special structure of planar curves allow us to establish the following main result of this work:
\begin{thm}
If $\Omega_0\supset\supset B_1$ is a bounded $C^{2,\alpha}$ domain satisfying the `convexity' condition (1.1), then the flow exists for all positive time.  Moreover,  $\Omega_t$ converges to $B_{R_{opt}}$ as $t\to\infty.$
\end{thm} 

That is, condition (1.1) is sufficient to rule out singularities, and the smooth flow has nice large time asymptotics.  Here $B_{R_{opt}}$ a ball centered at the origin with radius $R_{opt}$, whose value is given in the next section. It is the stationary solution to our flow and corresponds to the optimal configuration to an elliptic problem in Athanasopoulos-Caffarelli-Kenig-Salsa \cite{ACKS}. 

We believe ideas from this paper would open the door to many other interesting open problems related to this flow and other geometric flows involving nonlocal phenomena. 

This paper is organised as follows: in the second section we recall some geometric properties established in \cite{Yu}, and give some results on radially symmetric flows; in the third section we prove that condition (1.1) is preserved under the flow based on a scaling argument and the comparison principle; this is used in section four to show that our flow stays at a positive distance from $B_{R_{opt}}$ for any finite time, which leads to a uniform Lipschitz estimate of the curve based on special geometry of the plane; in the last section we use this regularity to get a smooth flow via parabolic regularisation. The large time convergence is also given in this last section.

\section{Preparatory results}
Throughout this paper we assume that the initial configuration $\Omega_0$ is a bounded $C^{2,\alpha}$ domain in $\mathbb{R}^2$ containing $B_1$ as a compact subset. Then result in \cite{Yu} guarantees the existence of a smooth flow of sets $\{\Omega_t\}_{0\le t<T}$ for some possibly short time $T$. For the definition of a smooth flow and the proof of this short-time well-posedness result, the reader should consult \cite{Yu}.

One of the remarkable geometric properties established in \cite{Yu} is the comparison principle for our flow:
\begin{thm}
Let $\Omega_0$ and $\Omega'_0$ be two initial domains with respective smooth flows $\{\Omega_t\}$ and $\{\Omega'_t\}$ up to time $T$. If $\Omega_0\subset\Omega'_0$ then $\Omega_t\subset\Omega'_t$ for all $0\le t\le T$.
\end{thm} 

It was shown in \cite{Yu} that at least for short-time we can parametrise our sets by a scalar function $\Lambda$ over the initial configuration by $$\partial\Omega_t=\{x+\Lambda(x,t)\nu_{\partial\Omega_0}(x):x\in\partial\Omega_0, 0\le t<T\}.$$ This $\Lambda$ then solves an equation of parabolic type. Under condition (1.1), there is a bijection of class $C^{2,\alpha}$ between $\Omega_0$ and the unit sphere $\mathbb{S}^1$, and the parametrisation for the domains passes to a function over $\mathbb{S}^1$, together with the equation:

\begin{prop}
The parametrisation $\Lambda$, as a function over $\mathbb{S}^1$, solves $$\partial_t\Lambda(\theta,t)=a(\theta,\Lambda,\partial_\theta\Lambda)\partial_{\theta\theta}\Lambda+b(\theta,\Lambda,\partial_\theta\Lambda)+c(\theta,\Lambda,\partial_\theta\Lambda)|\nabla u|^2 \text{ on $\mathbb{S}^1$},$$ where $a,b$ and $c$ are functions depending only on $\Omega_0$ with $C^{2,\alpha}$-regularity in their components. Furthermore, $0<c\le a<C<\infty$ for constants $c$ and $C$ only depending on $\Omega_0$. 
\end{prop}

Flows starting from balls are useful as barriers to our problem. We collect several facts about them:

\begin{prop}
Let $R_{opt}$ be the unique solution to $(\log R)^2=1/R,$ then $B_{R_{opt}}$ is the unique stationary solution to our flow. 

For general $R_0>1$, the flow with initial domain $B_{R_0}$ remains smooth for all positive time and is of the form $\{B_{R(t)}\}_{t\ge 0}$.

If $R_0> R_{opt}$ then $R(t)$ is decreasing with respect to $t$, $R(t)>R_0$ for all $t$, and $\lim_{t\to\infty}R(t)=R_{opt}.$ Similar results hold for the case when $1<R_0<R_{opt}$.\end{prop} 

\begin{proof}
For $R>1$, the solution to $$\begin{cases}\Delta u=0 &\text{ in $B_R\backslash \overline{B_1}$,}\\u=1 &\text{ in $B_1$,}\\u=0 &\text{ outside $B_R$}\end{cases}$$ is explicitly given by 

$$u_R(x)=-\frac{1}{\log R}(\log |x|-\log R) \text{ in $B_R\backslash B_1$},$$ and $$u_R(x)=0 \text{ outside $B_R$}.$$

In particular the normal derivative along $\partial B_R$ is $$(u_R)_{\nu}(R\frac{x}{|x|})=-\frac{1}{R\log R}.$$ Meanwhile the curvature of a ball of radius $R$ is $1/R$.

Thus the flow reduces to the following ODE on the radii: $$\begin{cases}\frac{d}{dt}R(t)=-\frac{1}{R}+\frac{1}{R^2(\log R)^2} &\text{ for $t>0$,}\\ R(0)=R_0 &\text{ at $t=0$}.\end{cases}$$

Note that the function $f(R)=-\frac{1}{R}+\frac{1}{R^2(\log R)^2}$ is strictly decreasing on $R>1$ and  has the unique root at $R_{opt}$, we see $B_{R_{opt}}$ is the unique stationary solution to the flow.  

For $R_0>R_{opt}$, the monotone property of $f$ implies $dR/dt<0$ thus $R$ is decreasing and remains larger than $R_{opt}$. This and the fact that the only singularity of $f(R)$ is at $R=1$ gives the global existence of the solution to the flow for the case when $R_0>R_{opt}$. The strict decreasing property of $f$ implies also the convergence result as $t\to\infty$. 

The case for $R_0<R_{opt}$ follows from symmetric argument. 
\end{proof} 

\section{Preservation of the `convexity' condition}
In this section we prove that condition (1.1) is preserved by the flow.  Since no explicit evolution equation can be written for the curvature or $u_\nu^2$, we cannot apply simple maximum principle type argument on curvature as it is often done for classical geometric flows.  Instead we combine a time-shift technique with the geometric comparison as stated in Theorem 2.1 to get the following:

\begin{thm}
Let $\{\Omega_t\}_{0\le t<T}$ be a smooth flow satisfying condition (1.1), then for all $t<T$ $$H-u_\nu^2> 0 \text{ on $\partial\Omega_t$}.$$
\end{thm} 

\begin{proof}
Since the initially the domain is strictly shrinking, by continuity and compactness, one finds $\tau>0$ such that $$H-u_\nu^2> 0 \text{ on $\partial\Omega_t$}$$ for all $t<\tau$. 

Define $w(x,t):= u(x,t+\tau)$ and $O_t=\{w(\cdot, t)>0\}$, then $w(\cdot,t)$ is the capacity potential of $O_t$ and the normal velocity of the boundary of $O_t$ is 
\begin{align*}
V_{O_t}&=V_{\Omega_{t-\tau}}(x)\\&=H_{\partial\Omega_{t-\tau}}-(u_{\Omega_{t-\tau}})_\nu^2\\&=H_{\partial O_t}-w_\nu^2.
\end{align*}Consequently $\{O_t\}$ is a flow with initial domain $O_0=\Omega_\tau$, which is contained in $\Omega_0$ by our `convexity' assumption. 

Therefore Theorem 2.1 ensures $O_t\subset\Omega_t$ for all $t$, and from here elliptic comparison gives $w(\cdot, t)\le u(\cdot, t)$. That is, 
$$u(\cdot, t+\tau)-u(\cdot, t)\le 0 \text{ for all small $\tau$}.$$ But this implies $u$ is decreasing and thus its positive phase $\Omega_t$ shrinking for all $t$, which is only possible if $$H-u_\nu^2> 0.$$
\end{proof} 

As a simple corollary, we have 
\begin{cor}
The domain $\Omega_t$ remains strictly convex. 
\end{cor}

\section{Positive distance to $B_{R_{opt}}$ and consequences}
In this section we prove that for any finite time, our domain remains at a positive distance to $B_{R_{opt}}$, where $R_{opt}$ is as in Proposition 2.2. This has strong regularity implications for our planar flow. 

\begin{prop}
Let $\{\Omega_t\}_{0\le t<T}$ be a smooth flow with $\Omega_0$ satisfying condition (1.1). Then $B_{R_{opt}}$ is compactly contained in $\Omega_t$ for all $t$ unless $\Omega_0=B_{R_{opt}}$.
\end{prop} 

\begin{proof}

Since condition (1.1) is preserved by the flow, it suffices to prove that any domain $\Omega$ with $H-u_{\nu}^2> 0$ along $\partial\Omega$ must contain $B_{R_{opt}}$ as a compact subset.  

For small $\lambda$, $$\lambda B_{R_{opt}} \subset\subset \Omega$$ by our assumption that $\Omega\supset B_1.$ Define $$\lambda_*:=\sup\{\lambda|\lambda B_{R_{opt}}\subset \Omega\},$$our goal is to show $\lambda_*>1$.

By definition of $\lambda_*$ we have \begin{equation}\lambda_* B_{R_{opt}}\subset\Omega,\end{equation} and $$\partial(\lambda_* B_{R_{opt}})\cap \partial\Omega\neq\emptyset.$$ Pick $x_0\in \partial(\lambda_* B_{R_{opt}})\cap \partial\Omega,$ then (4.1) implies $$H_{\partial\Omega}(x_0)\le H_{\partial(\lambda_* B_{R_{opt}})}(x_0).$$ Also elliptic comparison principle gives $$u_{\Omega}\ge u_{\lambda_*B_{R_{opt}}}.$$ This together with the fact that both potentials vanish at $x_0$ gives $$(u_\Omega)_\nu^2(x_0)>(u_{\lambda_*B_{R_{opt}}})_\nu(x_0)^2.$$
In conclusion, \begin{equation*}H_{\partial(\lambda_* B_{R_{opt}})}(x_0)-(u_{\lambda_*B_{R_{opt}}})_\nu(x_0)^2> H_{\partial\Omega}(x_0)-(u_\Omega)_\nu^2(x_0)> 0.
\end{equation*}

Now we plug in the explicit expression as obtained in Section 2 for the left-hand side to get $$\frac{1}{\lambda_*R_{opt}}-(\frac{1}{\lambda_*R_{opt}\log (\lambda_*R_{opt})})^2> 0=\frac{1}{R_{opt}}-(\frac{1}{R_{opt}\log (R_{opt})})^2.$$ From here the monotonicity of the function $f$ as defined in the proof of Proposition 2.2 gives $$\lambda_*> 1.$$
\end{proof} 

By comparing with radially symmetric flows we give a lower bound on the distance between $\partial\Omega_t$ and $B_{R_{opt}}$ depending only on initial data:

\begin{prop}
Let $\{\Omega_{t}\}_{0\le t<T}$ be a smooth flow with $\Omega_0$ satisfying condition (1.1). There exists a positive constant $\delta$ depending only on $T$ and $dist(\partial\Omega_0,B_{R_{opt}})$ such that $$dist(\partial \Omega_t, B_{R_{opt}})\ge\delta$$ for all $t< T$. 
\end{prop} 
\begin{proof}
The previous proposition ensures the positivity of $dist(\partial\Omega_0,B_{R_{opt}})$, and in particular one finds $\epsilon=\epsilon (dist(\partial\Omega_0,B_{R_{opt}}))$ such that $B_{R_{opt}+\epsilon}\subset\Omega_0$. 

Start a radially symmetric flow as in Proposition 2.2 with $$R_0=R_{opt}+\epsilon.$$ Theorem 2.1 guarantees that for all $t$ $$B_{R(t)}\subset\Omega_t.$$ And in particular $$dist(\partial\Omega_t,B_{R_{opt}})\ge R(t)-R_{opt}\ge R(T)-R_{opt}.$$ The right-hand side is a positive constant depending only on $T$ and $dist(\partial\Omega_0, B_{R_{opt}}).$
\end{proof} 

One consequence of the positive distance is uniform Lipschitz bound on the potential:
\begin{prop}
Let $\{\Omega_{t}\}_{0\le t<T}$ be a smooth flow with $\Omega_0$ satisfying condition (1.1). Then $$|u_{\nu}|\le \frac{1}{R_{opt}-1}.$$
\end{prop} 

\begin{proof}
Define $$d(x)=dist(x,\partial\Omega_t).$$ Convexity of the domains gives \cite{CC} $$\Delta d(x)\le 0 \text{ in $\Omega_t$}.$$ Meanwhile, $d=0=u \text{ on $\partial\Omega_t$}$ and $d\ge R_{opt}-1=(R_{opt}-1)u \text{  on $\partial B_1$}.$ Thus comparison principle gives $d\ge (R_{opt}-1)u $ inside $\Omega_t\backslash B_1$, and consequently $$(R_{opt}-1)|u_\nu|\le|\nabla d|=1.$$ 
\end{proof} 

Using the special geometry of the plane, the positive distance translates to Lipschitz regularity of the free boundary:

\begin{prop}
Let $\{\Omega_t\}_{0\le t<T}$ be a smooth flow with $\Omega_0$ satisfying condition (1.1). Then there is a finite constant $L$ depending only on the diameter of $\Omega_0$ and  $\delta$ as in Proposition 4.2 such that $\partial\Omega_t$  has Lipschitz norm less than $L$. 
\end{prop} 

\begin{proof}
Fix a time $t$, convexity of $\Omega_t$ implies that $\partial\Omega_t$ intersects with $\{r\theta|r>0\}$ exactly once for each $\theta\in\mathbb{S}^1$. Thus $\partial\Omega_t$ can be parametrised by the following function on the unit circle $$\rho(x/|x|)=|x|,$$ where $x\in\partial\Omega_t$.

From each $x\in\partial\Omega_t$, there are exactly two tangent lines to $B_{R_{opt}}$, $\overline{xp}$ and $\overline{xq}$. Here $\overline{AB}$ denotes the line segment between the points $A$ and $B$, and $p$ and $q$ are the two contact points between the tangent lines and $B_{R_{opt}}$.

Since $|x|\ge \delta+R_{opt}$, the angle between $\overline{op}$ and $\overline{oq}$ is bounded from below by $\theta_0=2\cos^{-1} (\frac{R_{opt}}{R_{opt}+\delta})$.  Here $o$ is the origin of the plane. 

Now note that convexity implies the segments $\overline{xp}$ and $\overline{xq}$ are contained in $\Omega_t$, thus for $y\in\partial\Omega_t$ with $|y/|y|-x/|x||\le\frac{1}{2}\theta_0$, $\overline{oy}$ intersects either $\overline{xp}$ or $\overline{xq}$ at some $z$. Hence \begin{align*}\rho(x/|x|)-\rho(y/|y|)&\ge |x|-|z|\\&\ge c(x/|x|-z/|z|)\\&= c(x/|x|-z/|z|).\end{align*} Here $c$ is the Lipschitz norm of the function that maps each point $\theta$ on $\mathbb{S}^1$ to the distance from $o$ to the intersection of the ray along $\theta$ with $\overline{xp}$ or $\overline{xq}$. Thus $c$ is bounded.

For $y\partial\Omega_t$ with $|y/|y|-x/|x||\ge\frac{1}{2}\theta_0$,  one can simply use the diameter bound and note that since our domains are shrinking, all diameters are bounded by the diameter of the initial domain. This gives \begin{align*}\rho(x/|x|)-\rho(y/|y|)&\le diam(\Omega_0)\\&=\frac{diam(\Omega_0)}{\theta_0}\frac{1}{\theta_0}\\&\le\frac{2diam(\Omega_0)}{\theta_0}|y/|y|-x/|x||. \end{align*} This gives the Lipschitz estimate for $\rho$ and hence on $\partial\Omega_t$.
\end{proof}

\section{Long term existence and asymptotics}
In this section we give the proof of the main result, namely, Theorem 1.1. Note that once we have the long term existence of the smooth flow, we can use the radially symmetric flow starting from $B_R$ for sufficiently large $R$ as a barrier flow, then Proposition 2.2 gives the convergence result.  Thus it suffices to prove the long term existence of the smooth flow. 

\begin{proof} (of the long term existence.)
Suppose $\{\Omega_t\}_{0\le t<T}$ is a maximal smooth flow with $\Omega_0$ satisfying condition (1.1). 

Suppose $T$ is finite, we claim that \begin{equation}\limsup_{t\to T}[\rho(\cdot,t)]_{Lip}=\infty,\end{equation}where $\rho(\cdot, t)$ is the function defined in the proof of Proposition 4.4.  

Suppose, on the contrary, that $[\rho(\cdot,t)]_{Lip}$ remains bounded as $t\to T$. Then Proposition 2.2 together with Proposition 4.3 show that $\rho$ solves a parabolic equation of bounded coefficients. Consequently $\partial_\theta\rho$ solves the divergence-type equation with bounded coefficients: $$\partial_t(\partial_\theta\rho)=\partial_\theta(a(\theta, \rho,\partial_\theta\rho)\partial_{\theta\theta}\rho+b(\theta,\rho,\partial_{\theta}\rho)+c(\theta,\rho,\partial_\theta\rho)u_\nu^2).$$ 

De Giorgi-Nash-Moser theory then gives uniform $C^{\alpha}$-estimate on $\partial_\theta\rho$. That is, $\rho$ is $C^{1,\alpha}$. As a result, $a,b$ and $c$ are all H\"older continuous. Moreover, $u$ is a harmonic function with smooth data inside a $C^{1,\alpha}$ domain, so it has H\"older continuous derivatives up to the boundary. Therefore, $\rho$ solves a parabolic equation with H\"older coefficients and hence it is smooth by Schauder theory, contradicting the finite terminal time $T$. This proves (5.1).

Therefore if we have a finite terminal time $T$, the Lipschitz norm of $\rho$ must blow up. This, however, has been ruled out by Proposition 4.4. Thus $T=\infty$ and no singularity can develop for our flow. \end{proof}

\section*{Acknowledgements}Hui Yu would like to thank many colleagues and friends, in particular Hongjie Dong, Dennis Kriventsov and Tianling Jin, for fruitful discussions concerning this project, especially for the discussion about parabolic equations in one spatial dimension.  He is also grateful to Yanyan Li and Jingang Xiong for their invitation to Beijing Normal University, where part of this work was conducted.




\begin{thebibliography}{100}

\bibitem{ACKS} I. Athanasopoulos, L. A. Caffarelli, C. Kenig, S. Salsa, {\em An area-Dirichlet minimization problem}, Comm. Pure Appl. Math. 54 (2001), no. 4, 479-499.

\bibitem{CC} L. A. Caffarelli, A. C\'ordoba, {\em An elementary regularity theory of minimal surfaces}, Diff. Integral Equations 6 (1993), 1-13.
\bibitem{GH} M. Gage, R. S. Hamilton, {\em The heat equation shrinking convex plane curves}, J. Differential Geom. 23 (1986), 69-96.
\bibitem{H} G. Huisken, {\em Flow by mean curvature of convex surfaces into spheres}, J. Differential Geom. 20 (1984) 237-266.
\bibitem{Yu} H. Yu, {\em Motion of sets by curvature and derivative of capacity potential}, preprint. 
\end{thebibliography}
\end{document}